\documentclass[11pt]{article}
\usepackage{epic,latexsym,amssymb}
\usepackage{color}
\usepackage{tikz}
\usepackage{amsfonts,epsf,amsmath}
\usepackage{lineno}
\usepackage{soul}

\usepackage{subfig}
\usepackage{epic}
\usepackage{amssymb}
\usepackage{latexsym}
\usepackage{tikz}
\usetikzlibrary{decorations.markings}
\usepackage{pgf}

\usetikzlibrary{arrows}

\usepackage{comment}

\newtheorem{theorem}{Theorem}

\newtheorem{lemma}{Lemma}

\newtheorem{definition}{Definition}

\newtheorem{corollary}{Corollary}

\newtheorem{conjecture}{Conjecture}

\newcommand{\qed}{$\Box$}
\newcommand{\QED}{$\Box$}

\let\oldenumerate\enumerate
\renewcommand{\enumerate}{
	\oldenumerate
	\setlength{\itemsep}{1pt}
	\setlength{\parskip}{0pt}
	\setlength{\parsep}{0pt}
}

\def\vertex(#1){\put(#1){\circle*{2}}}
\def\vertexo(#1){\put(#1){\circle{2}}}
\def\vert(#1){\put(#1){\circle*{1.5}}}
\def\verto(#1){\put(#1){\circle{1.5}}}
\def\lab(#1)#2{\put(#1){\makebox(0,0)[c]{#2}}}
\begin{document}

	\title{ $k$-Coalitions in Graphs}

	\author{  Abbas Jafari$^{1}$,	
		Saeid Alikhani$^{1}$,  
		Davood Bakhshesh$^2$ 
	}

	%\date{\today}
	
	\maketitle
	
	\begin{center}
		
		$^1$Department of Mathematical Sciences, Yazd University, 89195-741, Yazd, Iran\\
		
		$^{2}$Department of Computer Science, University of Bojnord, Bojnord, Iran\\ 
		
		\medskip
		{\tt  abbasjafaryb91@gmail.com~~ alikhani@yazd.ac.ir ~~  d.bakhshesh@ub.ac.ir}
	\end{center}

	\date{}
	\maketitle
	%\linenumbers

	\begin{abstract}
		In this paper, we propose and investigate the concept of $k$-coalitions in graphs, where $k\ge 1$ is an integer. A $k$-coalition refers to a pair of disjoint vertex sets that jointly constitute a $k$-dominating set of the graph, meaning that every vertex not in the set has at least $k$ neighbors in the set. We define a $k$-coalition partition of a graph as a vertex partition in which each set is either a $k$-dominating set with  exactly $k$ members or forms a $k$-coalition with another set in the partition. The maximum number of sets in a $k$-coalition partition is called the $k$-coalition number of the graph represented by $C_k(G)$. We present fundamental findings regarding the properties of $k$-coalitions and their connections with other graph parameters. We obtain the exact values of $2$-coalition number of some specific graphs and also study graphs with large $2$-coalition number. 
	\end{abstract}


	\section{Introduction}
	
	Consider a graph $G$ with vertex set $V = V(G)$, where we only consider graphs that are simple and undirected. Two vertices are said to be neighbors if they are adjacent. For an integer $k \ge 1$, a $k$-dominating set of $G$ is a set $S$ of vertices such that each vertex in $V \backslash S$ is adjacent to at least $k$ vertices in $S$. The smallest possible size of a $k$-dominating set of $G$ is referred to as the $k$-domination number of $G$, denoted by $\gamma_k(G)$. The interested reader may refer to~\cite{HaHeHe-20,HaHeHe-21} for a comprehensive overview of dominating sets in graphs.
	
	A coalition in a graph $G$ is a pair of sets $S_1$ and $S_2$ that are not dominating sets of $G$, but their union $S_1 \cup S_2$ is a dominating set of $G$. Such a pair forms a coalition and are coalition partners. A vertex partition $\mathfrak{X} = \{S_1, \ldots, S_k\}$ of the vertex set $V(G)$ is called a coalition partition of $G$ if every set $S_i \in \mathfrak{X}$ is either a dominating set of $G$ with cardinality $|S_i| = 1$, or not a dominating set but forms a coalition with some $S_j \in \mathfrak{X}$. The coalition number of a graph is the maximum number of sets in a coalition partition.
	
	The concept of a coalition in graphs was introduced by Haynes, Hedetniemi, Hedetniemi, McRae, and Mohan in \cite{coal0}. Their foundational studies have set the stage for much of the subsequent research on coalition numbers and coalition graphs. Notably, they explored upper bounds on coalition numbers, providing essential insights and bounds that help understand the maximum coalition number possible in various graph classes~\cite{Hay1}. Additionally, they developed the concept of coalition graphs, which are derived from the original graph by focusing on the coalition structure among the vertices, allowing for deeper analysis of the interactions and relationships within graph coalitions~\cite{Hay2}. Extending their previous work, they introduced self-coalition graphs, a specific type of coalition graph where the coalitions possess a self-referential property, adding another layer of complexity and applicability to the study of coalition graphs~\cite{Hay3}.
	
	Recent studies have continued to build upon these foundational concepts, expanding the scope and depth of coalition studies in graphs. Significant contributions in this area include the exploration of total coalitions, independent coalitions, connected coalitions, and specific investigations into coalition numbers in tree structures and singleton coalition graph chains. Alikhani et al. (2024) delve into total coalitions, providing a detailed analysis of how coalitions can encompass all vertices in a graph and the implications of such total structures. This study offers new metrics and bounds for total coalition numbers, expanding the understanding of coalition dynamics in comprehensive vertex sets~\cite{QM}. They also explore the independence properties within coalitions, defining and characterizing independent coalition graphs, leading to new theoretical insights and practical applications in graph theory~\cite{DMGT0}. Moreover, Alikhani et al. investigate connected coalitions, where coalitions form connected subgraphs, providing critical results on the connectivity aspects of coalitions, which are vital for applications requiring robust and resilient coalition structures~\cite{DMGT}.
	
	In addition, Bakhshesh, Henning, and Pradhan (2023) focus on tree structures, a fundamental graph class, to determine coalition numbers. Their findings offer specific insights and bounds applicable to trees, enriching the overall understanding of coalition numbers in hierarchical and acyclic graph structures~\cite{Davood}. In an upcoming publication, Bakhshesh, Henning, and Pradhan explore chains formed by singleton coalitions. This study provides a novel perspective on coalition structures by examining the sequential and chain-like properties of singleton coalitions, contributing to the broader theory of coalition graphs~\cite{Davood2}.
	
	Building on these established concepts, the exploration of $k$-coalitions in graphs represents a promising extension. A {\em $k$-coalition} consists of a pair of disjoint vertex sets that together form a $k$-dominating set of the graph, meaning that each vertex not in the set has at least $k$ neighbors within the set. We define a $k$-coalition partition of a graph as a vertex partition where each set is either a $k$-dominating set with exactly $k$ members or forms a $k$-coalition with another set in the partition. The maximum number of sets in a $k$-coalition partition is referred to as the {\em $k$-coalition number} of the graph, denoted by $C_k(G)$. This generalization has the potential to tackle more complex real-world problems, where entities often participate in multiple overlapping groups.
	\medskip

	In  the next section, after introducing $k$-coalition partition and $k$-coalition number, we present a sufficient condition for existence $k$-coalition partition, and we obtain bounds on the
	$k$-coalition number. Moreover, we study the $k$-coalition number of certain graphs such as complete graphs, trees, paths, cycles, and corona of  paths and cycles with $\overline{K_l}$. In Section 4, we study graphs with large $2$-coalition numbers and characterize trees $T$ of order $n$ with $C_2(T)=n$ and $C_2(T)=n-1$. Finally, we conclude
	the paper in Section 5.

	\section{Existence and some bounds}
	In this section, we present a sufficient condition for existence $k$-coalition partition, and also  we present some bounds on the
	$k$-coalition number.

	\begin{definition}
		Two sets $U_1\subseteq V$ and $U_2\subseteq V$ form  a $k$-coalition (are $k$-coalition partners) if neither is a $k$-dominating set, but their  union is a $k$-dominating set.  We define a $k$-coalition partition $\Theta=\{U_1,\ldots, U_r\}$ of a graph as a vertex partition in which each set of $\Theta$ is either a $k$-dominating set with  exactly $k$ members or forms a $k$-coalition with another set in the partition. We call the $k$-coalition number of a graph the maximum number of sets in a $k$-coalition partition denoted by $C_k(G)$. 
	\end{definition} 
	
	A domatic partition is a partition
	of the vertex set into dominating sets, in other words, a partition $\pi$ = $\{V_1, V_2, . . . , V_k \}$ of $V(G)$  
	such that every set $V_{i}$ is a dominating set in $G$. 
	Cockayne and Hedetniemi \cite{Cockayne} introduced the domatic number of a graph $d(G)$ as the maximum order $k$ of a vertex partition. For more details on the domatic number refer to e.g., \cite{11,12,13}. 
	Now,  we propose the notion of $k$-domatic number of $G$. 
	
	\begin{definition}
		A $k$-domatic partition is a partition
		of the vertex set into $k$-dominating sets, in other words, a partition $\pi$ = $\{V_1, V_2, . . . , V_k \}$ of $V(G)$  
		such that every set $V_{i}$ is a $k$-dominating set in $G$. 
		The $k$-domatic number of a graph $d_k(G)$ is the maximum order $k$ of a vertex partition. 	
	\end{definition}

	\begin{theorem}
		\label{thm1}
		For any integer $k\ge 1$ and  any graph $G$ with $\delta(G)\ge k$  there is a $k$-coalition partition.
	\end{theorem}
	\begin{proof}
		Consider a graph $G$ with a $k$-domatic partition $\Phi = \{X_1, \ldots, X_s\}$. Let $1\leq i<s$. Without loss of generality, assume that $X_i$ is a minimal $k$-dominating set of $G$. If it is not, then there exists a minimal $k$-dominating set $X_i'\subseteq X_i$. In this case, we replace $X_i$ with $X_i'$ and add all members of $X_i\backslash X_i'$ to $X_s$. To construct a $k$-coalition partition $\Theta$ of $G$, we split each minimal $k$-dominating set $X_i$ with $i<s$ into two non-empty sets $X_{i,1}$ and $X_{i,2}$ and add them to $\Theta$. If $k=1$ and $|X_i|=1$, we simply add $X_i$ to $\Theta$ without splitting it. Note that neither $X_{i,1}$ nor $X_{i,2}$ is a $k$-dominating set, but their union is a $k$-dominating set. Next, we consider the set $X_s$. If $X_s$ is a minimal $k$-dominating set, we split it into two non-empty sets $X_{s,1}$ and $X_{s,2}$ and add them to $\Theta$ to complete the construction. If $X_s$ is not a minimal $k$-dominating set, there exists a set $X_s'\subseteq X_s$ that is minimal and $k$-dominating. We split $X_s'$ into two non-empty sets $X_{s,1}'$ and $X_{s,2}'$ and add them to $\Theta$. Let $X_s''=X_s\backslash X_s'$. It is important to observe that $X_s''$ cannot be a $k$-dominating set, as this would imply that $d_k(G)>s$, which contradicts the fact that $\Phi$ is a $k$-domatic partition of $G$. If $X_s''$ forms a $k$-coalition with any set in $\Theta$, we add it to $\Theta$ and finish the construction. Otherwise, we remove $X_{s,2}'$ from $\Theta$ and add $X_{s,2}'\cup X_s''$ to $\Theta$. \QED
	\end{proof}

	The following statement gives a lower bound on $C_k(G)$ for connected graphs of order $n$  by means of the $k$-domatic number. 
	
	\begin{theorem} \label{gelemm}
		If $G$ is a connected graph, then $C_k(G)\ge 2d_k(G)$.
	\end{theorem} 
	
	\begin{proof} 
		Let $G$ have a $k$-domatic partition $\mathcal{C} = \{C_1, C_2, \ldots, C_s\}$ with $d_k(G) = s$. Without loss of generality, we assume that the sets $\{C_1, C_2, \ldots, C_{s-1}\}$ are minimal $k$-dominating sets. If any set $C_i$ is not minimal, we can find a subset $C'_i \subseteq C_i$ that is a minimal $k$-dominating set and add the remaining vertices to the set $C_s$. Note that if we partition a minimal $k$-dominating set with more than one element into two non-empty sets, we obtain two non-$k$-dominating sets that together form a $k$-coalition. Consequently, we divide each non-singleton set $C_i$ into two sets $C_{i,1}$ and $C_{i,2}$ that form a $k$-coalition. This results in a new partition $\mathcal{C}'$ consisting of non-$k$-dominating sets, each of which pairs with another non-$k$-dominating set in $\mathcal{C}'$ to form a coalition.
		Next, we consider the $k$-dominating set $C_s$. If $C_s$ is a minimal $k$-dominating set, we divide it into two non-$k$-dominating sets, add these sets to $\mathcal{C}'$, and obtain a $k$-coalition partition of cardinality at least $2s$. Since $s = d_k(G)$, it follows that $C_k(G) \geq 2d_k(G)$.
		If $C_s$ is not a minimal $k$-dominating set, we aim to find a subset $C'_s \subseteq C_s$ that is minimal. We then partition $C'_s$ into two non-$k$-dominating sets that together form a $k$-coalition. Let $C''_s$ be the complement of $C'_s$ in $C_s$, and append $C'_{s,1}$ and $C'_{s,2}$ to $\mathcal{C}'$. If $C''_s$ can merge with any non-$k$-dominating set to form a $k$-coalition, we can obtain a $k$-coalition partition of cardinality at least $2s+1$ by adding $C''_s$ to $\mathcal{C}'$. Thus, $C_k(G) \geq 2d_k(G) + 1$. 
		However, if $C''_s$ cannot form a $k$-coalition with any set in $\mathcal{C}'$, we remove $C'_{s,2}$ from $\mathcal{C}'$ and add the set $C'_{s,2} \cup C''_s$ to $\mathcal{C}'$. This results in a $k$-coalition partition of cardinality at least $2s$. Therefore, $C_k(G) \geq 2d_k(G)$.
		
		Based on the above arguments, we conclude that $C_k(G) \geq 2d_k(G)$, completing the proof.
		\qed
	\end{proof}

	Using similar arguments as in the proof of Theorem \ref{gelemm}, we obtain the following corollary:
	
	\begin{corollary}
		For even $k$, $C_k(G)\geq d_{k/2}(G)$. 
	\end{corollary}

	\begin{lemma}
		\label{l2delta}
		For any graph $G$ and any $c_k$-partition $\cal C$ of $G$, any set of  ${\cal C}$ forms a $k$-coalition with at most $\Delta(G)-k+2$  sets of ${\cal C}$.
	\end{lemma}
	\begin{proof}
		Consider a vertex $v$ in graph $G$ and a set $S \in \mathcal{C}$ such that $v \in S$. If $S$ is a $k$-dominating set, then by definition, $S$ forms a $k$-coalition with no other set in $\mathcal{C}$, thereby confirming the result. Now, suppose $S$ is not a $k$-dominating set. Then, there exists a vertex $x \notin S$ that is not $k$-dominated by $S$. For any set $A \in \mathcal{C}$ that does not include $x$ and forms a $k$-coalition with $S$, in order for $A \cup S$ to $k$-dominate vertex $x$, the set $A \cup S$ must include at least $k$ vertices from $N(x)$ (the neighborhood of $x$). Let $x \in A$. To maximize the number of sets that form a $k$-coalition with $S$, the set $S$ must contain at most $k-1$ neighbors of $x$, leaving the remaining neighbors of $x$ to be covered by all coalition partners of $S$ except $A$. Therefore, in the worst case, $S$ forms a $k$-coalition with at most $1 + |N(x)| - (k-1) \leq \Delta(G) - k + 2$ sets. This completes the proof.
		\QED
	\end{proof}
	\begin{theorem}
		\label{lthm}
		For any graph $G$ with the  maximum degree $\Delta(G)$ and $k>\delta(G)$, 
		$ 	C_{k}(G)\le \Delta(G)-k+3$.
	\end{theorem}
	
	\begin{proof}
		Let $x$ be a vertex of $G$ of degree $\deg(x)=\delta(G)$. Let $\cal C$ be a $c_{k}$-partition of $G$  of the  cardinality  $C_{k}(G)$. Let $X\in {\cal C}$ such that $x\in X$. If  $N(x)\subseteq X$, then any set of ${\cal C}\setminus X$ must  form a $k$-coalition only with $X$. Hence, by Lemma~\ref{l2delta}, $C_{k}(G)\leq 1+\Delta(G)-k+2= \Delta(G)-k+3$.  Now, assume that  $N(x)\nsubseteq X$. Let $A\neq X$ and $B\neq X$ be two sets of ${\cal C}$. If $A$ and $B$ forms a $k$-coalition, then $A\cup B$ is a $k$-dominating set. Since $x\not\in A\cup B$, $x$ must have at least  $k$ neighbors in $A\cup B$, which is a contradiction because $x$ has $\delta(G)<k$ neighbors. Hence, every set of ${\cal C}$ must only form a $k$-coalition with $X$. Hence, by    Lemma~\ref{l2delta}, we have $C_k(G)\leq \Delta(G)-k+2+1=\Delta(G)-k+3$. \QED
	\end{proof}

	\begin{theorem}
		\label{lthm}
		For any graph $G$ with $\Delta(G)\geq \delta(G)+1$, $C_{\delta(G)}(G)\le 2\Delta(G)-2\delta(G)+4$.
	\end{theorem}
	
	\begin{proof}
		Let $k=\delta(G)$ and let $x$ be a vertex of $G$ of degree $\delta(G)$. Let $\mathcal{C}$ be a 	$c_k$-partition of $G$ of the cardinality $C_k(G)$. Let $X\in \mathcal{C}$ such that $x\in X$. 
		\begin{itemize}
			\item 
			If $N(x)\cap X\neq \emptyset$, then, any set of $\mathcal{C}\setminus X$
			must form a $k$-coalition only with $X$. Hence, by Lemma \ref{l2delta},
			$C_k(G)\leq \Delta(G)-k+3$. Since $\Delta(G)\geq \delta(G)=k$, we have $C_{\delta(G)}(G)\le 2\Delta(G)-2\delta(G)+4$. 
			
			\item If $N(x)\cap X= \emptyset$, then we consider the following cases. 
			
			\begin{itemize}
				
				\item  There exists at least three sets $A,B$ and $C$ in $c_k$-partition $\mathcal{C}\setminus \{X\}$ which have intersect with  $N(x)$, i.e.,  $A\cap N(x)\neq \emptyset$, $B\cap N(x)\neq \emptyset$ and $C\cap N(x)\neq \emptyset$.  In this case for every two partners $S_i$ and $S_j$ which forms $k$-coalition, we have $S_i=X$, or $S_j=X$. Therefore by Lemma \ref{l2delta}, $C_k(G)\leq \Delta(G)-k+3$.

				\item There exists exactly two  sets $A$ and $B$ with $A\cap N(x)\neq \emptyset$, $B\cap N(x)\neq \emptyset$. If $A$ and $B$ form a $k$-coalition, then $N(x)\subseteq A\cup B$.  Hence,  there is no sets $C\setminus\{X\}$ forming $k$-coalition with $A$ or $B$. Hence, by Lemma \ref{l2delta} the set $X$ is in at most $\Delta(G)-k+2$ $k$-coalition and therefore $C_k(G)\leq 1+\Delta(G)-k+2+2=\Delta(G)-k+5$. Since $\Delta(G)\ge k+1$, we have $C_{k}(G)\le 2(\Delta(G)-k+2)$.  If $A$ and $B$ do not form a $k$-coalition, then each of $A$ and $B$ form a $k$-coalitions with $X$, then  by Lemma \ref{l2delta}, $C_k(G)\leq1+\Delta(G)-k+2=\Delta(G)-k+3$. Since $\Delta(G)\ge k+1$, we have $C_{k}(G)\le 2(\Delta(G)-k+2)$.
				
				\item There exist exactly one set $A\in {\cal C}$ with  $A\cap N(x)\neq \emptyset$.
				If $N(x)\nsubseteq A$, then for $k$-dominating of vertex $x$, $A$ does not form a $k$-coalition with any set ${\cal C}\setminus \{X\}$. Hence $A$ and $X$ form a $k$-coalition. Hence, by Lemma \ref{l2delta}, $C_{k}(G)\leq \Delta(G)-k+3$. Since $\Delta(G)\ge k+1$, we have $C_{k}(G)\le 2(\Delta(G)-k+2)$. Now, suppose that $N(x)\subseteq A$.  Then, if $A$ and $X$ form a $k$-coalition, then by Lemma \ref{l2delta}, each of the sets $A$ and $X$ form  $k$-coalitions with at most $\Delta(G)-k+2$ sets. Since we assumed that $A$ and $X$ form a $k$-coalition, then $C(G)\leq 2(\Delta(G)-k+1)+2=2(\Delta(G)-k+2)$. Now, assume that $A$ and $X$ do not form a $k$-coalition. Then, $A\cup X$ is not a $k$-dominating set. Hence, there exists a vertex $w$ which not $k$-dominated by $A$ and $X$.  Note that $A\cup X$ contains no vertex of $N[w]$. Let $\mathcal{P}=\{S\in {\cal C}| N[w]\cap S\neq \emptyset\}$. It is clear that ${\cal C}=\{A,X\}\cup {\cal P}$. Since $|{\cal P}|\leq \Delta(G)-k+2$, we have $C_k(G)\leq 2+\Delta(G)-k+2=\Delta(G)-k+4\leq2(\Delta(G)-k+2)$ (since $\Delta(G)\ge k+1$).\QED
			\end{itemize}
		\end{itemize}
	\end{proof} 
	
	In the following, we show that the $k$-coalition number of any $k$-regular graph is  $3$ or $4$.  
	
	\begin{corollary} 
		If $G$ is a $k$-regular graph, then $3\leq C_k(G) \leq 4$. 
	\end{corollary} 
	%##############
	\begin{proof} 
		Suppose that two vertices $v_1$ and $v_2$ in $G$ are adjacent. Then  
		\[ 
		\big\{ 
		V\setminus\{v_1,v_2\}, \{v_1\}, \{v_2\} \big\}
		\] 
		is a $k$-coalition partition  of $G$ and so $C_k(G)\geq 3$. 
		Now let $\mathcal{C}$ be a
		$c_k$-partition of $G$ with  the cardinality $C_k(G)$. If there exists set $X$ in $\mathcal{C}$ such that contains two adjacent vertices, then for some vertex $x$ in $X$,
		$N(x) \cap X \neq \emptyset$. By Lemma \ref{l2delta}, $C_k(G)\leq \Delta(G)-k+3=3$. Therefore in this case $C_k(G)=3$. 
		
		If  no sets in $\mathcal{C}$ contain two adjacent vertices, then we consider two following cases:
		
		Case 1. If there exists a vertex  $x$ such that $N(x)$ has intersect with 
		just one set  or more than two set in $\mathcal{C}$, 
		then by the proof of Theorem \ref{lthm}, 
		\[ 
		C_k(G) \leq \max \big\{\Delta(G) -k + 4, ~ 2(\Delta(G)-k+2)\big\}=4. 
		\]

		Case 2. If there exists a vertex $x$ such that  $N(x)$ has intersection with two sets in $\mathcal{C}$, say $S_1$ and $S_2$. Again, we consider two cases:
		\begin{itemize}  
			\item
			
			If $S_1$ or $S_2$ form a $k$-coalition with $X$, then by Lemma \ref{l2delta}
			\[
			C_k(G)\leq 1+\Delta(G)-k+2+1 =\Delta(G)-k+4=4. 
			\]
			\item 
			
			If $X$ is a $k$-dominating set, we are done. Now  let $X$ forms a $k$-coalition with $S_0$. Since any vertex $v_0\in N(x)$ (which are not in $S_0\cup X$) is dominated by $S_0\cup X$, and the neighborhood of any vertex cannot be in the only one set in a partition, so  there is a vertex $v_0\in N(X)$ such that is adjacent to a vertex in $S_0$. So  $X$ just form $k$-coalition with $S_0$ and so $\mathcal{C}=\{X,S_0,S_1,S_2\}$. 
			Therefore we have the result. \qed 
		\end{itemize} 
	\end{proof}

	\section{$k$-coalition number of specific graphs}
	
	First let recall the definition of corona of two graphs. By taking a single instance of graph $F$ and $|V(F)|$ instances of graph $H$, and
	linking the $i$-th vertex of $F$ to each vertex in the $i$-th instance of $H$, we obtain a graph denoted as $F\circ H$. This graph  consider as the corona product of $F$ and
	$H$.

	In this section, we study the $k$-coalition number of certain graphs, such as complete graph, trees, path $P_n$, cycle $C_n$, $P_n\circ \overline{K_l}$ and $C_n\circ \overline{K_l}$. We start with complete graphs.

	\begin{theorem}
		For every $2\leq k\leq n-1$, $C_k(K_n)=n-k+2$.
	\end{theorem}
	\begin{proof}
		It is easy to see that the partition $\pi=\{\{v_1,v_2,...,v_{k-1}\},\{v_k\},\{v_{k+1}\},...,\{v_n\}\}$ is a $k$-coalition partition of $K_n$. \qed
	\end{proof} 
	
	The following theorem gives a lower bound for the $k$-coalition number of complete bipartite graph. 
	
	\begin{theorem}
		If $K_{s,t}$ is a complete bipartite graph and $s\leq t$, then $C_k(K_{s,t})\geq t-k+2.$
	\end{theorem}
	\begin{proof}
		Suppose that $X=\{v_1,v_2, \dots,v_s\}$ and $Y=\{v'_1,v'_2,\dots,v'_t\}$ are two parts of $K_{s,t}$. By considering the following partition we have the result:
		\[ 
		\Big\{{
			\{v_1,v_2,\dots, v_{s},v'_1,v'_2,\dots,v'_{k-1}\},\{v'_k\},\dots,\{v'_{t-1}\},\{v'_t\}
		}\Big\}.
		\]\qed
	\end{proof}

	Using Theorem \ref{lthm}, we have the following result. 
	\begin{corollary}
		For every tree $G$, 	$C_2(G)\leq \Delta(G)+1$.
	\end{corollary}
	
	\begin{corollary}
		For every $k\in \mathbb{N}$, there exists a tree $T$ with maximum degree $k$ and 
		$C_2(T)=k+1$. 
	\end{corollary}
	\begin{proof}
		Consider the star $K_{1,k}$ and pendant two vertices to each leaves of it (Figure \ref{fs1}). The $2$-coalition number of this tree is $k+1$. \qed
	\end{proof}

	\begin{figure}[ht]
		\begin{center}
			\includegraphics[width=0.6\linewidth]{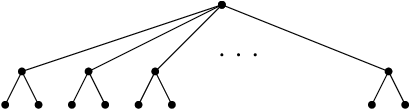}
			\caption{A tree $T$ with maximum degree $k$ and 
				$C_2(T)=k+1$. }
			\label{fs1}
		\end{center}
	\end{figure}

	Now, using Theorem \ref{lthm}, we prove the following result. 
	\begin{theorem}
		For any path $P_n$
		\begin{equation*}
		C_2(P_n)=\left\{
		\begin{aligned}[c]
		1 &\text{~~~} n=1,2 \\
		2 & \text{~~~}n=3 \\
		3 & \text{~~~}n\geq 4.
		\end{aligned}\right.
		\end{equation*}
	\end{theorem}
	\begin{proof}
		It is easy to verify that form $n\leq 3$, $C_2(P_1)=1$, $C_2(P_2)=1$, and $C_2(P_3)=2$. Now, assume that the path $P_n$ with  $n\geq 4$.  By Theorem \ref{lthm}, for any path $P_n$ we have $C_2(P_n)\leq \Delta(P_n)-2+3=3$. Now, we have a 2-coalition partition of cardinality  3 for $P_n$ as follows.
		$$\left\{\{v_1,v_n\},  \{v_{2i}|1<2i< n\}, \{v_{2i+1}|1<2i+1<n\}\right\}.$$ \qed
	\end{proof} 
	
	To obtain the $2$-coalition number of cycles, we need the following easy lemma: 
	\begin{lemma}\label{l1}
		If  $S \subseteq V(C_n)$ is  a $2$-dominating set of $C_n$, then $|S|\geq \dfrac{n}{2}$.
	\end{lemma}
	\begin{proof} 
		For every  $v \in V\backslash S$, $\deg_{S}(v)=2$ and 
		$  \sum_{v\in V\backslash S} \deg_{S}(v)=2(n-|S|)$. So by definition, 
		\[ 
		2(n-|S|) \leq \sum_{v\in  S} \deg(v) = 2|S|,
		\]
		and so  $	\dfrac{n}{2} \leq |S|$. \qed 
		
	\end{proof}
	
	\begin{theorem}
		If $\Theta = \{S_0, S_1, \dots , S_t\}$ is  a $2$-coalition partition for $C_n$ and every  $S_i \in \Theta $ does not  contain consecutive vertices of $C_n$, then    
		$n$   is even and $C_2(C_n)=4$.
	\end{theorem}
	
	\begin{proof}
		We consider two cases: 
		
		Case 1:  	
		There is a vertex $v_0 \in S_0$ such that is adjacent to different partition $S_1$ and $S_2$. 
		Every $S_i \in \Theta\backslash\{S_0,S_1,S_2\} $ form a $2$-coalition with $S_0$. So 
		\[
		\dfrac{n}{2} \leq |S_i\cup S_0|=|S_i|+|S_0|.
		\]
		
		Since  $\deg_{V\backslash\{S_1,S_2\}}(v_0)=0$, so $S_1$ and $S_2$ form a $2$-coalition. Therefore  
		\[ \dfrac{n}{2}+\dfrac{n}{2} \leq |S_i|+|S_0|+|S_1|+|S_2|\leq n=|V|, \]
		and so 
		\[
		|S_i|+|S_0|=\dfrac{n}{2},~~|S_1|+|S_2|=\dfrac{n}{2}.
		\]
		Therefore $n$ is even and $|\Theta|=4$.
		
		Case 2: We do not have any vertex adjacent to different partition. So the vertices must be alternatively  in the same partition $S_0$. Therefore $n$ is even and $S_0$ is $2$-dominating set. Hence  $|S_0|=2$ and $n=4$. \qed 
		
	\end{proof}

	\begin{lemma}
		The	$2$-coalition number of odd cycles is $3$. 
	\end{lemma}
	\begin{proof}
		Suppose that $\Theta=\{S_1,S_2,\dots, S_t\}$  is a $2$-coalition of $C_n$, where  $n>4$ is odd. Then some $S_i \in \Theta$ contain two consecutive vertices. 
		
		Suppose that $V(C_n)=\{v_1,v_2,...,v_n\}$ and $E(C_n)=\big\{\{v_1,v_2\},\{v_2,v_3\}, ...,\{v_{n-1},v_n\},\{v_n,v_1\}\big\}$. 
		Without loss of generality, suppose that $v_1, v_2$ are in $S_i$. Since $\deg_{V\setminus S_i}(v_1)\leq 1$ so  every $S_j \in \Theta\setminus S_i$ form a $2$-coalition with $S_i$. 
		Note that   $S_i$ cannot be a $2$-dominating set, since $S_i$  is element of the $2$-coalition partition. Therefore $V\backslash S_i$ contain consecutive vertices. Assume that $V_1,V_2,...,V_r$ are  subsets of $V\backslash S_i$ such that their vertices are consecutive. 
		Every $S_j \in \Theta \backslash S_i$ and $S_i$ form a $2$-coalition. If $S_j\cup S_i \subsetneq V$ then the vertices  of  $V_q$ ($1\leq q\leq r$)
		must be alternately contained in $S_i$ and there exist just $2$ partition with this conditions. So $|\Theta|\leq 3.$
		
		For every $C_n$ with $n\geq 4$, we have a $2$-coalition partition with 3 element as follows:
		\[\big\{ \{v_1\},  \{v_{2i}|2\leq 2i\leq n\}, \{v_{2i+1}|2\leq 2i+1\leq n\}    \big\}.\]
		Therefore we have the result. 	
		\qed   
	\end{proof}

	\begin{corollary}
		The $2$-coalition number of $C_n$ is:
		\[ 
		C_2(C_n)=\begin{cases}
		1 & n=1 \\
		4& n  \text{ is even}\\
		3& n  \text{ is odd}.
		\end{cases}
		\]
	\end{corollary}

	Using Theorem \ref{lthm}, we obtain the following result.
	\begin{corollary}\label{th2}
		For any cycle $C_n$ and path $P_n$, we have $C_2(C_n \circ K_1)=4$ and $C_2(P_n \circ K_1)=4$.
	\end{corollary}
	\begin{proof}
		Let $V'$ contains vertices of degree one and $V''=V\setminus V'$.  
		By Theorem \ref{lthm}, we have $C_2(C_n \circ K_1)\le4$ and $C_2(P_n \circ K_1)\le4$.
		Now, we present a $2$-coalition partition with $4$ elements for $P_n$ and $C_n$, as follows:
		\[ 
		\big\{	 V\setminus\{v_{n-2},v_{n-1},v_n\},\{v_{n-2}\},\{v_{n-1}\},\{v_n\}\big\} 
		\]
		such that $v_{n-2},v_{n-1},v_n \in V''$. \qed
	\end{proof}

	\begin{theorem}\label{cofc} 
		The $k$-coalition number of $C_n\circ \overline{K_l}$ is:
		\[
		C_k(C_n \circ\overline{K_l})=
		\begin{cases}
		2 &l \leq k-3\\
		3 &l = k-2\\
		4 & l=k-1\\
		2 & l\geq k
		\end{cases}
		\]
	\end{theorem}
	
	\begin{proof}
		Let  $\Theta=\{S_1,S_2,\dots,S_t\}$ be a $k$-coalition partition and $S_i$ and $S_j$ form a $k$-coalition. Suppose that $V'\setminus\{v| \deg(v)=1\}$ and $V''=\{v|\deg(v)=l+2\}$. We consider two cases: 
		
		Case 1: $l \leq k-3$. Since the degree of any vertices is less than $k$, so $S_i \cup S_j $ must contain whole of the vertices. Therefore $\Theta=\{S_i,S_j\}$.
		
		Case 2: $l = k-2$. Since $\deg(v)=1<k$ for $v \in V'$, so $V'\subset S_i\cup S_j$. If $V'$ intersect with both $S_i$ and $S_j$, it is clear that $|\Theta|=2$. Now let  $V'\subset S_i$. Since $S_i$ is not a $k$-dominating set, so there are at least $2$ consecutive  vertices $v_1$ and $v_2$ in $V\setminus S_i$. Every $S_t\in \Theta \setminus \{S_i\}$ must form a $k$-coalition with $S_i$
		and contains at least one of $v_1,v_2$. Therefore $|\Theta|<=3$. We can have a $k$-coalition with 3 elements as follows:
		\[ 
		\big\{	 V\setminus\{v_{n-1},v_n\},\{v_{n-1}\},\{v_n\}\big\} \quad 
		\text{such that }v_{n-1},v_n \in V''
		\]

		Case 3: $ l=k-1$. Similar to the proof of Corollary \ref{th2}.
		
		Case 4: $ l\geq k$. In this case the vertices can be partition into two disjoint sets $V'$ and $V''$ such that $V'\{v| \deg(v)=1\}$ and $V''=\{v|\deg(v)=l+2\}$. Let $S_i$ and $S_j$ be a $k$-coalition, so $V'\subset S_i \cup S_j$. If $V'\subset S_i$ then $S_i$ is a $k$-dominating set, since $V\setminus S_i \subset V''$ and $\deg_{V'}(v)=l+2>k$ for $v \in V''$, and it  contract the definition of $k$-coalition.
		Therefore sets $V'\cap S_i $ and $V' \cap S_j$ are not empty and it  implies  $\Theta=\{S_i,S_j\}$. \qed

	\end{proof}

	With similar proof of Theorem \ref{cofc} we have the following result: 
	\begin{theorem}
		The $k$-coalition number of $P_n\circ \overline{K_l}$ is:
		\[
		C_k(P_n \circ\overline{K_l})=
		\begin{cases}
		2 &l \leq k-3\\
		2 &l = k-2\quad n\leq 3\\
		3 &l = k-2\quad n\geq4\\
		4 & l=k-1\\
		2 & l\geq k
		\end{cases}
		\]
	\end{theorem}

	\section{Graphs with large $2$-coalition number} 
	Characterization of graphs of order $n$ whose coalition number is $n$ or $n-1$ is an interesting subject, see \cite{DMGT,Davood}. 
	In this section, we study graphs with large $2$-coalition number.

	\begin{theorem}\label{th9}
		If $C_2(G)=n$, then for a vertex $v$, $\deg(v)\geq n-2$.
	\end{theorem}
	
	\begin{proof}
		Since $C_2(G)=n$, so for every $v_1  \in V(G)$, there is $v_2 \in V(G)$ such that form a $2$-coalition. Therefore $v_1$ and $v_2$ must be adjacent to all vertices of $V(G)\setminus\{v_1,v_2\}$. \qed
	\end{proof}
	
	\begin{lemma}
		If  $C_2(G)=n$  and $\deg(v)=n-2$, then there is only one $2$-coalition partner for $v$. %of $G$ is the vertex $v$ with  a non-adjacent  vertex $v'$.
	\end{lemma}
	
	\begin{proof}
		Since $C_2(G)=n$, so there is a vertex $v'$ that form a $2$-coalition with $v$. Therefore all vertices of $V(G)\setminus\{v,v'\}$ must be adjacent to $v$. Since $\deg(v)=n-2$, so $v'$ is not adjacent to $v$. \qed
	\end{proof}

	\begin{lemma}\label{lm3}
		For any even number $n$, there is an $(n-2)$-regular graph $H$ with $C_2(H)=n$.
	\end{lemma}
	
	\begin{proof}
		Suppose that  $V(H)=\{v_1,v_2,\dots,v_n\}$. For each $ i$, let   $v_{2i-1}$ and $v_{2i}$ are not  adjacent and these two vertices  are adjacent to all vertices in  $V(H)\setminus\{v_{2i-1},v_{2i}\}$. Obviously this graph $H$ is $(n-2)$-regular and has $2$-coalition partition $ \big\{\{v_1\},\{v_2\},...,\{v_n\}\big\}$. Therefore we have the result. \qed
		
	\end{proof}
	
	\begin{corollary}
		If $G$ is an $(n-2)$-regular graph with $C_2(G)=n$, then $n$ is even and $G$ is isomorphic to graph $H$ in the proof of  Lemma \ref{lm3}. 
	\end{corollary}
	
	\begin{corollary}
		If $C_2(G)=n$ and $n$ is odd, then  number of  full vertices of $G$ is odd.
	\end{corollary}

	\section{Trees with large $2$-coalition number}
	Characterization trees of order $n$ whose coalition number is $n$ or $n-1$ is an interesting subject. In this section, we study trees with large $2$-coalition number. 
	Here, we obtain another upper bound for the $2$-coalition number of trees. 
	
	\begin{theorem}\label{bound}
		For any tree $T$ of order $n$, $C_2(T)\leq \frac{n}{2}+1$.
	\end{theorem} 
	\begin{proof}
		Let $\pi$ be a $c_2$-partition of $T$ and $L$ be the set of all leaves of $T$. We know that $L\subset \pi$.  Any set $X\neq L$ in $\pi$ forms a $2$-coalition with $L$. By Theorem \ref{lthm}, there are at most 
		$\Delta(T)-2+3=\Delta(T)+1$ sets in $\pi$. But it is easy to see that for any graph with $k$ leaves, $\Delta(T)\leq k$, and so $$C_2(T)\leq k+1.$$ On the other hand, since there are $k$ leaves in $L$, so we have at most $n-k$ vertices which are not in $L$. If any set $X\neq L$ in $\pi$ is a singleton (worst case), then $$C_2(T)\leq 1+n-k.$$
		From these two upper bounds we have $k+1=1+n-k$ and so $n=2k$. Therefore we have the result. \qed
	\end{proof}

	\begin{corollary} Let $T$ be a tree of order $n$.
		\begin{enumerate} 
			\item[(i)] 
			If $C_2(T)=n$, then $T=P_2$.
			\item[(ii)]
			If $C_2(T)=n-1$, then $T=P_4$. 
		\end{enumerate}
	\end{corollary}
	\begin{proof}
		\begin{enumerate} 
			\item[(i)] 
			Suppose that $C_2(T)=n$. By Theorem \ref{bound}, we have $n\leq \frac{n}{2}+1$, so $n\leq 2$. Therefore $T=P_2$. 
			\item[(ii)] 
			If 	$C_2(T)=n-1$, then by Theorem \ref{bound}, we have $n-1\leq \frac{n}{2}+1$, so $n\leq 4$. Therefore $T=P_4$.
		\end{enumerate}  
	\end{proof}

	\begin{theorem}\label{treed}
		If there is a vertex $x$ in the tree $T$ such that the distance between $x$ and all leaves of $T$ is at least $2$, then   $C_2(T)\geq 3$. 
	\end{theorem} 
	\begin{proof}
		Suppose that  $S=\{v\in V(T)|d(v,x)\geq 2\}$ and  $S_1=\{w\in V(T)|d(v,w)=1\}$. Then $S$, $S_1$ and $\{x\}$ is a $c_2$-partition. Therefore we have the result. \qed  	
	\end{proof}

	Now, we show that using Theorem \ref{treed} we have another proof for $2$-coalition number of paths.

	\begin{corollary} 
		For any $n\geq 5$, $C_2(P_n)=3$.  
	\end{corollary} 
	\begin{proof}
		By Theorem \ref{lthm}, $C_2(P_n)\leq \Delta(P_n)-l+3=3$. On the other hand 
		by Theorem \ref{treed}, for $n\geq 5$,  $C_2(P_n)\geq 3$. Therefore we have the result. \qed
	\end{proof}

	\section{Conclusion} 
	
	This paper introduces the concept of the $k$-coalition in
	graphs and investigates some  properties related to $k$-coalition number. We have 
	presented a sufficient condition for existence $k$-coalition partition, and also we
	presented  some bounds on the $k$-coalition number.  Utilizing these bounds, we have determined the precise values of $k$-coalition number of some specific graphs. We studied the graphs $G$ with large $ C_k(G)$. We have outlined some unresolved problems
	and potential research directions related to the $k$-coalition number of graphs.
	Also, there is still much work to be done in this area.
	
	\begin{enumerate} 
		\item[1.]  We proved that for $s\leq t$,	$C_k(K_{s,t})\geq t-k+2$. We think that the following conjecture is true. 
		\begin{conjecture}
			For $s\leq t$,	$C_k(K_{s,t})= t-k+2$.
		\end{conjecture}
		
		\item[2.] What is the exact values of $k$-coalition number of specific graphs, such as path, cycle, tree and unicyclic graphs for $k\geq 3$.
		
		\item[3.] Study Nordhaus and Gaddum lower and upper bounds on the sum and the
		product of the $k$-calition number of a graph and its complement.
		
		\item[4.] What is the $k$-coalition number of graph operations, such as corona, Cartesian product, join, lexicographic, and so on?
		
		\item[5.] Associated with 	every $k$-coalition partition $\pi$ of a graph $G$, there is a graph called the $k$-coalition graph of $G$
		with respect to $\pi$, denoted $kCG(G,\pi)$, the vertices of which correspond one-to-one with the sets $V_1, V_2,..., V_k$  of $\pi$ and two vertices are adjacent in $kCG(G,\pi)$ if and only if
		their corresponding sets in $\pi$ form a coalition. Study of $k$-coalition graph is an interesting subject.  
		
	\end{enumerate}

\end{document}